\documentclass[12pt]{article}
\usepackage{amssymb, amscd, amsmath, stackrel}
\usepackage{comment}

\newtheorem{theorem}{Theorem}
\newtheorem{lemma}[theorem]{Lemma}
\newtheorem{cor}[theorem]{Corollary}
\newtheorem{prop}[theorem]{Proposition}

\newcommand{\llbracket}{[[}
\newcommand{\rrbracket}{]]}

\title{Obstructions to 
Deformation Quantization of Bundles}
\author{Vladimir Baranovsky, Greg Huey}
\date{March 10, 2026}

\begin{document}

\maketitle

\begin{center}
    ABSTRACT
\end{center}

Let $\left(M, \mathcal{O}_M \right)$ be a smooth algebraic variety over field $\kappa$ of characteristic $0$ with an algebraic symplectic form $\omega$, or a complex manifold with a holomorphic form $\omega$. Furthermore, let $E$ be a vector bundle over $\left(M, \mathcal{O}_M \right)$ and $\mathcal{O}_{\hbar}$ a deformation quantization of 
$\mathcal{O}_M$ compatible with $\omega$. 

Assuming that $E$ possesses a deformation quantization to order $\hbar^k$
we consider the problem of extending it to  order $\hbar^\ell$ for 
$\ell > k$,
and establish triviality of an obstruction class as a necessary condition for this extension to exist. Furthermore, in the case  $\ell \le 2k+1$, we prove that this condition is also sufficient.

\section{Introduction}

\subsection{Setup.}

Consider a smooth algebraic variety ($M$, $\mathcal{O}_M$) of dimension $2n$ 
over a field $\kappa$ of characteristic 
zero with an algebraic symplectic form $\omega$, or a complex manifold 
($M$, $\mathcal{O}_M$) of complex dimension $2n$ with 
a holomorphic form $\omega$.
We also assume that $M$ is equipped with a deformation quantization $\mathcal{O}_\hbar$
compatible with $\omega$, i.e. a sheaf (in Zariski, resp. analytic topology) of 
associative $\kappa\left[\left[\hbar\right]\right]$-algebras which is complete and separated in $\hbar$-adic
topology and equipped with an isomorphism $\mathcal{O}_\hbar/\hbar \mathcal{O}_\hbar
\simeq \mathcal{O}_M$ such that for two local sections $f, g$ of $\mathcal{O}_\hbar$
one has
$$
f * g - g * f  \equiv \hbar \omega \left(\overline{f}, \overline{g} \right) \ mod\ \hbar^2
$$
where $\overline{f}, \overline{g}$ are images of $f, g$ in $\mathcal{O}_M$, respectively. 

Given a vector bundle $E$ (we use the same notation for the locally free sheaf
of its algebraic, resp. holomorphic sections) and $k \in \left[0, \infty \right)$, its 
order $k$ deformation quantization is a  sheaf $E_k$ of locally free 
$\mathcal{O}_\hbar/ \hbar^{k+1} \mathcal{O}_\hbar$-modules $E_k$ with an 
isomorphism $E_k/ \hbar E_k \simeq E$ as sheaves of $\mathcal{O}_M$-modules. 
For $k = \infty$ we define an infinite order deformation quantization as 
a locally free sheaf $E_\hbar$ of $\mathcal{O}_\hbar$-modules with a similar
isomorphism $E_\hbar/ \hbar E_\hbar \simeq E$ as sheaves of $\mathcal{O}_M$-modules.
If we have an order $\ell$ quantization $E_{\ell}$ and an order $k$ quantization $E_k$
for $\ell > k$, we say that $E_{\ell}$ is an extension of $E_k$ if an isomorphism of sheaves of modules 
$E_{\ell}/\hbar^{k+1} E_{\ell} \simeq E_k$ is fixed. 

The goal of this paper is to study cohomological obstructions to existence of 
an order $\ell$ quantization extending an existing order $k$ quantization. 
The situation is easier for $\ell < 2k+2$, when the vanishing of a certain cohomology
class (depending on the choice of $E_k$) gives a necessary and sufficient condition. 
For $\ell \geq 2k+2$ (this includes the case $\ell = \infty$) this becomes a nonabelian 
cohomology problem. However, one may still study its abelian ``shadows", i.e. 
provide cohomology classes the vanishing of which is necessary (but not sufficient)
for existence of $E_{\ell}$.

\bigskip
\noindent
This paper is organized as follows. In Section 2 for 
$k +1 \leq \ell \leq 2k+1$
we define a sheaf $End(E)_{(k, \ell]}$ and state Theorem 1, 
asserting that the second cohomology of this sheaf contains
the obstruction to extending an order $k$ quantization 
to an order $\ell$ quantization, and whenever such obstruction 
vanishes, the first cohomology parameterizes all such extentions. 

For $\ell \geq 2k+2$ we define a similar sheaf $End_{\mathcal{O}_\hbar}^+
(E_k)$, an extension of $End_{\mathcal{O}_\hbar}
(E_k)$ by $\mathcal{O}_M$, and show that its second cohomology 
contains an obstruction class for the problem of extending $E_k$
to order $\ell$ for $\ell \geq 2k+2$. We state the main results of 
the paper, Theorems 1 and 2, regarding obstructions to extending
order $k$ quantizations to order $\ell$ quantizations. 

In Section 3 we recall some basic definitions related to 
Harish-Chandra pairs, relative Lie algebra cohomology and 
introduce the main object of study, the Harish-Chandra pair given by 
automorphisms and derivations of the pair $(\mathcal{D}, \left(\mathcal{D}/
\hbar^{k+1} \mathcal{D}\right)^{\oplus e})$, where $\mathcal{D}$ is 
the formal version of Weyl algebra and $e$ is the rank of $E$. The
reason for looking at the above pair is the formal Darboux lemma, claiming
that the pair $(\mathcal{O}_\hbar, E_k)$ becomes isomorphic 
(non-canonically) to $(\mathcal{D}, \left(\mathcal{D}/
\hbar^{k+1} \mathcal{D}\right)^{\oplus e})$, after formal completion 
at any point $x \in M$. We remind the argument in Section 4, see
Theorem \ref{completion}. Putting together the spaces of \textit{all} such isomorphisms
for various $x$
produces a torsor over $M$, over a proalgebraic group, and this torsor
comes with an additional \textit{transitive Harish Chandra  structure}
which implies that any associated vector bundle of the torsor has
a flat algebraic connection. In a particular case, taking flat sections
of such a connection gives us $E_k$, hence the quantization may
be recovered from the torsor. Using the general criterion for lifts of 
transitive Harish Chandra torsors, as established in 
\cite{BezKal03} we prove the main results (Theorems 1 and 2) at the 
end of Section 4. 

Section 5 explains our results in the language of Fedosov-type
connections on the jet bundle of $E_k$, which admits a
$C^\infty$ isomorphism with a standard bundle constructed from
$E$ and a cotangent bundle of $M$. Section 6 outlines some 
standard special cases. 

In Section 7 we treat the problem of lifting a morphism 
$\varphi_k: E_k \to F_k$ (as sheaves of $\mathcal{O}_\hbar$-modules) 
to a morphism $E_\ell \to F_\ell$ of order $\ell$ quantizatons. 
Similarly to the above, there is an obstruction 
$$
O
\in 
H^1\left(M, \mathcal{H}om_{\mathcal{O}_\hbar}\left(E_{\ell-k-1}, F_{\ell-k-1}\right) \right)
$$
 to existence
of such a morphism. It may happen that $O$ does not
vanish but adjusting the choice $E_\ell, F_\ell$ (as extensions of 
corresponding order $k$ quantizations) leads to a vanishing obstruction 
class and hence existence of the morphism $\varphi_\ell$ of 
adjusted order $\ell$ quantizations, lifting
$\varphi_k$. We are using a more standard approach of spectral sequences
since there can be no formal Darboux lemma for $\varphi_k$:
morphisms between free $\mathcal{O}_k$-modules cannot be reduced to 
a single local form.

\bigskip
\noindent
\textbf{Acknowledgements.} We thank A. Gorokhovsky for useful discussions.

\bigskip\bigskip

\section{Coefficient sheaves for obstructions}

In this section we assume that an order $k$ quantization $E_k$ is
given and we will define a sheaf giving a cohomology group that contains
the obstruction to extending $E_k$ to $E_{\ell}$ for $\ell > k$. This obstruction 
sheaf depends on the choice of $E_k$ although it has a filtration by 
subsheaves with associated graded quotients $End(E)$. Eventually, 
this leads to a
spectral sequence in which the second page is formed by cohomology of
$End(E)$ but differentials do depend on the choice of $E_k$. We consider
two cases:
\medskip
\noindent
\textbf{$k+1 \leq \ell \leq 2k+1$} and \textbf{$ \ell > 2k+1$}. 

\medskip
\noindent
For the case \textbf{$k+1 \leq \ell \leq 2k+1$}, define the 
sheaf $End_{\mathcal{O}_\hbar}\left(E_k\right)$ of endomorphisms 
of $E_k$ over $\mathcal{O}_\hbar$ and then set
$$
End\left(E\right)_{\left(k, \ell\right]} = \hbar^{k+1} \cdot \big[End_{\mathcal{O}_h}\left(E_k\right)/ \hbar^{s+1}
 End_{\mathcal{O}_h}\left( E_k \right)\big] \simeq \hbar^{k+1}  \cdot
End_{\mathcal{O}_\hbar} \left( E_s \right)
$$
where $s = \ell - k - 1 \in  \{0, \ldots, k\}$ and 
$E_s := E_k/\hbar^{s+1} E_k$. 
Here, and in similar notation below $\left( \hbar^{k+1} \cdot \ldots \right)$ does not 
denote the action of $\hbar^{k+1}$, it simply a placeholder to remind
that eventually $\hbar^{k+1} \cdot 
End_{\mathcal{O}_\hbar} \left( E_s \right)$ will be identified with 
a subsheaf of $End_{\mathcal{O}_\hbar} \left( E_{s+k+1=\ell} \right)$
formed by all multiplies of $\hbar^{k+1}$, while as a sheaf it is isomorphic to 
$End_{\mathcal{O}_\hbar} \left( E_s \right)$. 

The quotient
defined above is sheaf in the Zariski (or complex analytic) topology.
Note that it depends on $E_k$ and $s$ only, although there is $\ell$ in the
notation. The factor $\hbar^{k+1}$ in front of $End$ will be clarified
later, but it particular it implies the zero commutator bracket. 
Under the above restrictions on $\ell$ the main result of our paper 
can be stated as follows
\begin{theorem} For $k+1 \leq \ell \leq 2k+1$,
given an order $k$ quantization $E_k$ of a locally free sheaf 
$E$, there is a obstruction class in $H^2\left(M, End\left(E\right)_{\left(k, \ell\right]}\right)$
which vanishes if and only if $E_k$ admits an extension to an 
order $\ell$ quantization. If this is indeed the case, the set
of isomorphism classes of such extensions is a torsor over
$H^1\left(M, End\left(E\right)_{\left(k, \ell\right]}\right)$.
\end{theorem}

\medskip
\noindent
Next, consider the case \textbf{$\ell \ge 2k+2$}. 
In Section \ref{alltheaction}  we define a sheaf
of vector spaces 
$End_{\mathcal{O}_\hbar}\left(E_k\right)^+$ which fits into a short exact sequence
\[
0 \rightarrow \hbar^{2k+2}\cdot \mathcal{O}_M  \rightarrow End_{\mathcal{O}_\hbar} \left( E_k \right)^{+} \rightarrow  End 
\left(E \right)_{\left(k, 2k+1\right]}  \rightarrow 0
\]
Existence of a lift of $E_k$ to an order $\ell$ quantization 
 implies that the obstruction class in $H^2 \left( M , End_{\mathcal{O}_\hbar} \left( E_k \right)^{+} \right)$ vanishes.  

\begin{theorem} For $\ell \geq 2k+2$, 
given an order $k$ quantization $E_k$ of a locally free sheaf 
$E$, there is an obstruction class in $H^2 \left(M, End_{\mathcal{O}_\hbar} \left( E_k \right)^{+}\right)$
which vanishes if  $E_k$ admits an extension to an 
order $\ell$ quantization. 
\end{theorem}

\section{Lie algebras, extensions and cocycles}

\subsection{Extension of Lie algebras}

Consider the formal
Weyl algebra, a quotient of the algebra of non-commutative power series: 
\[
\mathcal{D}=\kappa\left\langle\langle x_{1},\ldots,x_{n},y_{1},\ldots,y_{n},\hbar\right\rangle\rangle /\left\{ y_{j}x_{k}-x_{k}y_{j}=\delta_{jk}\hbar\right\} 
\]
Its quotient $\mathcal{A}=\mathcal{D}/\hbar\mathcal{D}$ is isomorphic
to the algebra of power series $\kappa\left\llbracket x_{1},\ldots,x_{n},y_{1},\ldots,y_{n}\right\rrbracket $. Much of our discussion below will be related to the $\mathcal{D}$-module $\mathcal{M}_k = \left(\mathcal{D}/\hbar^{k+1} \mathcal{D} \right)^e$ (we fix the value of $e$ throughout the paper, 
eventually equal to the rank of the bundle $E$ to 
be quantized, and do not reflect
it in the notation), and to the Lie algebra $\mathfrak{g}_k$ of infinitesimal automorphisms of the 
pair $\left(\mathcal{D}, \mathcal{M}_k\right)$, i.e. pairs of linear maps compatible with natural filtrations 
$$
\psi: \mathcal{D} \to \mathcal{D}, \qquad \varphi: \mathcal{M}_k \to \mathcal{M}_k
$$
where $\psi$ is a Lie derivation and $\varphi$ satisfies $\varphi\left(xm\right) = \psi\left(x\right) m + x 
\varphi\left(m\right)$. It follows from the definitions that 
$$
\mathfrak{g}_{k} :=  Der\left(\mathcal{D}\right) \ltimes  \mathfrak{gl}\left(e, \mathcal{D}/\hbar^{k+1}
\mathcal{D}\right)  
$$
where $Der\left(\mathcal{D}\right)$ is the algebra of Lie derivations compatible with filtrations; 
which
acts on $\mathfrak{gl} \left(e, \mathcal{D}/\hbar^{k+1} \mathcal{D}\right) $ through its action on 
$\mathcal{D}$.

\bigskip
\noindent
Assuming $k+1 \leq \ell \leq 2k+1$ we have a Lie algebra extension with 
\textit{abelian} kernel:

\begin{equation}
0\to\mathfrak{h}\left(k,\ell\right)\to\mathfrak{g}_{\ell}\to\mathfrak{g}_{k}\text{\ensuremath{\to0}}
\label{eq:Lie_alg_SES}
\end{equation}
where, setting $\mathfrak{h}_k = 0 \ltimes \hbar^{k+1} \mathfrak{gl} \left(e, \mathcal{D} \right)$
we define $\mathfrak{h} \left(k, \ell \right)
: = \mathfrak{h}_k/\mathfrak{h}_\ell$.
Note that this has a non-trivial
structure of a module over $\mathfrak{g}_k$. 
 In the case when $\ell \geq 2k+2$ consider the quotient $\mathfrak{h}^+\left(k,\ell\right)$ of the above extension by the commutator subalgebra $[\mathfrak{h}\left(k,\ell\right) , \mathfrak{h}\left(k,\ell\right)]$ of 
$\mathfrak{h}\left(k,\ell\right)$ which is a Lie ideal in $\mathfrak{g}_{\ell}$. This 
leads to an abelian extension
\begin{equation}
0\to\mathfrak{h}^+\left(k,\ell\right)\to\mathfrak{g}'_{\ell}\to\mathfrak{g}_{k}\text{\ensuremath{\to0}}
\label{eq:Lie_alg_SES-prime}
\end{equation}
It is a straightforward computation that the commutator of 
$\mathfrak{gl}(e, \mathcal{D})$ is the kernel of the composition 
$$
\mathfrak{gl} \left(e, \mathcal{D} \right) \to \mathfrak{gl} \left(e, \mathcal{A} \right) 
\to \mathcal{A}
$$
where $\mathcal{A}$ is the abelian quotient introduced above and 
the last arrow is the trace map. In particular, the commutator contains
$\hbar \mathfrak{gl}\left(e, \mathcal{D}\right)$. Hence $\mathfrak{h}^+ \left(k, \ell \right)$ 
is actually independent on the
choice of $\ell \geq 2k+2$ and we denote it
simply by $\mathfrak{h}^+ \left(k\right)$. There exists
a nonsplit short exact sequence of $\mathfrak{g}_k$-modules: 
\[
0 \rightarrow \hbar^{2k+2} \cdot \mathcal{A}  
\rightarrow \mathfrak{h}^+ \left(k \right) \rightarrow  
\hbar^{k+1} \mathfrak{gl} \left( e , \mathcal{D} \right)/ \hbar^{2k+2} \mathfrak{gl} \left( e , \mathcal{D} \right)  \rightarrow 0
\]
To see why this is the case, first observe that the Lie algebra $ 
\mathfrak{g}_{k+1}$ acts on $\mathfrak{h} \left(k, 2k+2 \right)$.
The induced action on $\mathfrak{h}^+ \left(k \right)$ (= the quotient
by commutators in $\mathfrak{h} \left(k, 2k+2 \right)$) is trivial
on the kernel of the surjection $
\mathfrak{g}_{k+1} \to 
\mathfrak{g}_{k}$. In fact, this kernel is isomorphic to the 
abelian Lie algebra on  the vector space
$\hbar^{k+1} \cdot \mathfrak{gl} \left(e, \mathcal{A} \right)$ and the Lie 
action by this vector space takes values in the subspace
spanned by the commutators of $\mathfrak{h} \left(k, 2k+2 \right)$.

The following theorem follows immediately from the fact that the
extensions \eqref{eq:Lie_alg_SES} and \eqref{eq:Lie_alg_SES-prime} admit vector space splittings
$$
\mathfrak{g}_\ell \simeq \mathfrak{g}_k \oplus \mathfrak{h}\left(k, \ell\right), \qquad
\mathfrak{g}'_{2k+2} \simeq \mathfrak{g}_k \oplus \mathfrak{h}^+(k)
$$which are 
invariant with respect to the reductive subalgebra 
$\mathfrak{sp}\left(2\right) \oplus \mathfrak{gl}\left(e\right)$ in $\mathfrak{g}_k$. By the standard construction
the splittings define Lie algebra cocycles
$$
C_\ell: \Lambda^2 \mathfrak{g}_k \to \mathfrak{h}\left(k, \ell\right), \qquad 
C': \Lambda^2 \mathfrak{g}_k \to \mathfrak{h}^+ \left(k\right)
$$

\begin{theorem}
\label{thm:class_C_in_ker_nu}
For $k+1 \leq \ell \leq 2k+1$, the extension 
cocycle $C_\ell$ is relative with respect to the 
subalgebra $\mathfrak{sp}\left(2n, \kappa\right)\oplus\mathfrak{gl}\left(e, \kappa \right)$ and it gives an element in the kernel of 
$$
H^{2}\left(\mathfrak{g}_{k}, 
\mathfrak{sp}\left(2n, \kappa\right)\oplus\mathfrak{gl}\left(e, \kappa\right);\mathfrak{h}(k, \ell)\right)\to H^{2}\left(\mathfrak{g}_{\ell}, \mathfrak{sp}\left(2n, \kappa\right)\oplus \mathfrak{gl}\left(e, \kappa\right);\mathfrak{h}(k, \ell) \right)
$$
For $\ell \geq 2 k + 2$, a similar statement holds for the cocycle $C'$ with coefficients
in 
$\mathfrak{h}^+\left(k\right)$.
\end{theorem}

\subsection{Harish-Chandra pairs and cohomology}
Below we would like to work with objects slightly different from 
usual Lie algebras, which we now proceed to define.

\bigskip
\noindent
\textbf{Definition.} A Harish-Chandra pair $\left(\mathbf{F}, 
\mathfrak{g} \right)$ consists of a  proalgebraic group $\mathbf{F}$ with a Lie algebra 
$\mathfrak{f}$ and an additional 
Lie algebra $\mathfrak{g}$ with an embedding 
$\mathfrak{f} \subset \mathfrak{g}$. 
It is further assumed that $\mathfrak{g}$ has a $\mathbf{F}$-action 
compatible with the adjoint action of $\mathbf{F}$ on $\mathfrak{f}$. 

A module $V$ over such a pair is a module over $\mathbf{F}$ with 
a Lie action of $\mathfrak{g}$ on $V$ which is compatible with 
the tangential action of $\mathfrak{f}$ on $V$, and the action 
of $\mathbf{F}$ on $\mathfrak{g}$. 

\bigskip

In our examples we will always have a splitting 
$\mathbf{F} \simeq \mathbf{U} \rtimes \mathbf{K}$ with prounipotent 
$\mathbf{U}$ and $\mathbf{K}$ is a finite dimensional algebraic 
group acting by automorphisms on $\mathbf{U}$. 
In our case $\mathbf{K}$ will be product of a symplectic and a general linear
 group, in particularly reductive.
We work with 
 a complex $C^\bullet(\mathfrak{g}, \mathbf{K}; V)$ of relative Lie cochains
 $\Lambda^u \mathfrak{g} \to V$ which are invariant with respect 
 to the subgroup $\mathbf{K}$.  When $\mathbf{K}$ is reductive and connected, 
 $\mathbf{K}$-invariant cochains 
 can be identified with $\mathfrak{k}$-invariant
 cochains, for $\mathfrak{k} = Lie \left(\mathbf{K} \right)$.

In this setting the cohomology of
Lie Algebra $\mathfrak{g}$, relative to subgroup group $\mathbf{K}
\subset \mathbf{F}$, with coefficients in the left $\left(\mathbf{F}, \mathfrak{g} \right)$-module $V$, can be defined via the invariant version of 
the Chevalley\textendash Eilenberg
complex: 
\begin{equation}
C^{\ell}\left(\mathfrak{g},\mathbf{K};V\right):=Hom_{\kappa}\left(\wedge^{\ell}(\mathfrak{g}/\mathfrak{k}),V\right)^\mathbf{K}
\end{equation}
For $\omega\in C^{\ell}\left(\mathfrak{g},\mathbf{K};V\right)$ and 
$a_{j}\in\mathfrak{g}$
 the Chevalley\textendash Eilenberg differential
$\delta_{Lie}$ satisfies
\begin{equation}
\begin{array}{ccccc}
 & \left(\delta_{Lie}\omega\right)\left(a_{1},\cdots,a_{\ell+1}\right) & = & \stackrel[j=1]{\ell+1}{\sum}\left(-1\right)^{j+1}a_{j}\omega\left(a_{1},\cdots,\widehat{a_{j}},\cdots,a_{\ell+1}\right)+\\
 &  &  & \stackrel[j=1]{\ell+1}{\sum}\stackrel[i=1]{j-1}{\sum}\left(-1\right)^{j+i}\omega\left(\left[a_{i},a_{j}\right],a_{1},\cdots,\widehat{a_{i}},\cdots,\widehat{a_{j}},\cdots,a_{\ell+1}\right).
\end{array}\label{eq:diff_LA_CE}
\end{equation}
Note that
$i_x \omega = 0$ and $i_x \circ \delta_{Lie} \omega = 0$ whenever $x$ is a vector tangent to $\mathbf{K}$.
 The cohomology groups are defined in the usual way
\begin{equation}
H^{\ell}\left(\mathfrak{g},\mathbf{K},V\right)=\frac{\ker\left(\delta_{Lie}:C^{\ell}\rightarrow C^{\ell+1}\right)}{Im\left(\delta_{Lie}:C^{\ell-1}\rightarrow C^{\ell}\right)}
\end{equation}

\bigskip
\noindent 
The usual relation between extensions and second cohomology classes
generalizes to this case. 
We consider an abelian extension of Harish-Chandra pairs
\begin{equation}
\label{extensionHC}
1 \to \left(V, V \right) \to \left(\tilde{\mathbf{F}}, \tilde{\mathfrak{g}}\right) 
\to \left(\mathbf{F}, \mathfrak{g}\right) \to 1
\end{equation}
where the kernel $V$ is given by a Harish-Chandra module over
$\left(\mathbf{F}, \mathfrak{g} \right)$ and we view it as abelian algebraic 
group with the additive structure. We assume that the extension 
splits over a subgroup $\mathbf{K} \subset \mathbf{F}$.

\begin{theorem}
  Any abelian extension as above, splitting over a subgroup $\mathbf{K}$, defines
  a relative cohomology class in the kernel of 
$$
H^{2}\left(\mathfrak{g}, 
\mathbf{K};V\right)\to H^{2}\left(\tilde{\mathfrak{g}}, 
\mathbf{K}; V \right)
$$
\end{theorem}

\subsection{Main example}

The main example is as follows. For $k \geq 0$, the Harish-Chandra 
pairs is given by $\left(\mathbf{F}_k, \mathfrak{g}_k \right)$
where 
$\mathbf{F}_k = Aut \left(\mathcal{D}, \mathcal{M}_k\right)$ is the group 
of $\kappa\llbracket \hbar \rrbracket$-linear automorphisms 
of the pair $\left(\mathcal{D}, \mathcal{M}_k\right)$ which 
 preserve the filtration by powers of maximal ideal in $\mathcal{D}$, 
 and $\mathfrak{g}_k$ was defined in the beginning of this section.
 Note that $\mathfrak{f}_k = Lie\left(F_k\right)$ is naturally a subalgebra of 
 $\mathfrak{g}_k$ but latter is strictly larger since
 derivations $\frac{1}{\hbar}\left[x_i, \cdot\right]$ do not preserve the 
 maximal ideal of $\mathcal{D}$ generated by $x_i, y_j$ and $\hbar$.
If $k +1 \leq \ell  \leq 2k+1$, we
have an abelian extension of Harish-Chandra pairs 
\begin{equation}
\label{extention-kl}
1 \to \left(\mathfrak{h} \left( k, \ell \right), \mathfrak{h} \left(k, \ell \right) \right) \to \left( \mathbf{F}_\ell, \mathfrak{g}_\ell \right) \to \left( \mathbf{F}_k, \mathfrak{g}_k \right) \to 1
\end{equation}

\begin{cor}
\label{vanishing_cocycles}
For $k +1 \leq \ell  \leq 2k+1$, the extension \eqref{extention-kl}
splits over the reductive subgroup $\mathbf{K} = Sp\left( 2n, \mathbb{C} \right) 
\times GL\left(e, \mathbb{C}\right)$ and hence its extension cocycle is $\mathbf{K}$-invariant. It belongs to the kernel of 
$$
H^{2}\left(\mathfrak{g}_{k}, 
\mathbf{K};\mathfrak{h}(k, \ell)\right)\to H^{2}\left(\mathfrak{g}_{\ell}, \mathbf{K};\mathfrak{h}(k, \ell) \right)
$$
\end{cor}

\section{Harish-Chandra torsor approach.}

\subsection{Transitive Harish-Chandra torsors, their lifts and Gelfand-Fuchs map.}

A general notion of Harish-Chandra torsors can be found in \cite{BezKal03} 
but in this paper we only need its transive version. A \textit{transitive Harish-Chandra} 
torsor over a Harish-Chandra pair  $\left(\mathbf{F}, \mathfrak{g}\right)$ on a
complex manifold or an algebraic scheme $M$ is an $\mathbf{F}$-torsor $\pi: P \to M$ and 
an $\mathbf{F}$-equivariant Lie morphism $\mathfrak{g} \to \Gamma\left(P, T_P\right)$
extending the morphism of $\mathfrak{f}$ to global vector fields along the fibers; 
and such that the induced bundle morphism $\mathcal{O}_P \otimes_\kappa \mathfrak{g} \to T_P$ 
of locally free sheaves on $P$, is an isomorphism. Note that this implies that each 
point $p \in P$ of the tangent space is identified with $\mathfrak{g}$ and that the 
quotient $\mathfrak{g}/\mathfrak{f}$ projects isomorphically onto the tangent
space of $z = \pi\left(p\right)$. 

\bigskip
\noindent
For a transitive Harish-Chandra torsor $P$, we can apply two constructions. 
First, as explained in \cite{BezKal03}, for any $ \left(\mathbf{F}, \mathfrak{g}\right)$-module 
$V$ we can form an associated vector bundle $V_P$ on $M$ viewing $P$ as a usual 
$\mathbf{F}$-torsor, while the transitive Harish-Chandra structure on $P$ gives $V_P$ 
a flat connection. In the examples relevant to us, the sheaf of flat sections of $V_P$  
defines a deformation quantization $\mathcal{O}_\hbar$ of the structure sheaf, or an order 
$k$ deformation quantization $E_k$ of a vector bundle. 

Then $V \mapsto V_P$ gives
 a functor from the category of $ \left(\mathbf{F}, \mathfrak{g}\right)$-modules to the
 category of bundles on $M$ with a flat algebraic connection.
 Since the Lie algebra cohomology group $H^u\left(\mathfrak{g}, V\right)$
 computes $Ext^u \left(\kappa, V\right)$ in the category of such modules, it
 maps to $Ext^u \left(\kappa_P \simeq \mathcal{O}, V_P\right)$ in the 
 category of bundles with a flat connection, i.e. the
 de Rham cohomology group $H^u_{DR} \left(M, V_P\right)$. A minor modification
 of this construction involving relative Lie cohomology will 
 produce cohomology classes of endomorphism algebras of 
 quantized bundles.

\bigskip
\noindent
 In more detail: the transitive Harish-Chandra structure converts any $\ell$-cochain 
 $\alpha: \Lambda^{\ell} \mathfrak{g} \to V$ to an $\ell$-form on $P$
 with  values in $V$ (since every tangent space is identified with 
 $\mathfrak{g} $). This gives a morphism of vector spaces 
 $$
 C^{\ell}  \left(\mathfrak{g}; V \right) \to \Gamma \left(P, \Omega^{\ell}_P \otimes_\kappa V \right)
 $$
 which commutes with the differentials, as the value of the de Rham differential $d_{dR}$  at a set of vector fields is given by the Cartan homotopy
 formula matching the formula of the Lie algebra differential $\delta_{Lie}$. For a subgroup
 $\mathbf{K} \subset \mathbf{F}$ its restriction to $C^\ell \left(\mathfrak{g}, \mathbf{K}; V \right) 
 \subset C^\ell \left(\mathfrak{g}; V \right)$ takes 
 values in the subcomplex of
 $\mathbf{K}$-basic forms, that is the forms that are $\mathbf{K}$-invariant (for connected $\mathbf{K}$ this 
can be rephrased as $L_{\widetilde{\mathfrak{k}}} \alpha = 0$  all $\mathfrak{k}\in Lie\left( \mathbf{K} \right) $, where $\widetilde{\mathfrak{k}}$ is the vector field induced by $\mathfrak{k}$) and horizontal ($i_{\widetilde{\mathfrak{k}}} \alpha = 0$, i.e. $\alpha$ vanishes on any vector that is in the kernel of the differential of the projection map $\pi : P \rightarrow M$), cf. 
Chapter 9 in \cite{GS}. Note 
that the Lie derivative along a vector field induced by $\mathfrak{k}$ involves the action of the Lie algebra on the module $V$ (i.e. the trivial connection on $\mathcal{O}_P \otimes_\kappa V$
is adjusted by the endmonorphism valued 1-form obtained from the action morphism 
$\mathfrak{g} \to End_\kappa(V)$).  This defines 
the \textit{Gelfand-Fuchs} map 
$$
GF \left(P \right): C^{\bullet} \left(\mathfrak{g}, \mathbf{K}; V \right) \to \Gamma \left(M, \pi_*  \left(\Omega_P^\bullet \otimes_\kappa
  V \right)^{\mathbf{K}-basic} \right) 
$$
and we use the same notation $GF$ for the induced map on cohomology.

As $V$ is a Harish-Chandra module over $ \left(\mathbf{F}, \mathfrak{g}\right)$ and $P$
is a transitive Harish-Chandra torsor over the same pair, the
associated vector bundle $V_P$ over $M$ has a canonical flat connection. 
We denote by $\Omega^\bullet(V_P)$ the corresponding de Rham complex.
When $\mathbf{F} = \mathbf{K} \ltimes \mathbf{U}$ with prounipotent $\mathbf{U}$, it will be proved below in 
Theorem \ref{jetsandtorsors} that there exists a quasi-isomorphism of
sheaves of complexes
 \begin{equation}
 \label{eq:quasi-1}
  \Omega^\bullet_M \left(V_P \right) \to \pi_* \left(\Omega_P^\bullet \otimes_\kappa
  V \right)^{\mathbf{K}-basic}
 \end{equation}
 
Taking cohomology we obtain the Gelfand-Fuchs map: 
$$
GF: H^\bullet \left(\mathfrak{g}, \mathbf{K}; V \right) \to H^\bullet_{DR} \left(M, V_P \right)
$$
By Section 2.3 of \cite{BezKal03}, for a Harish-Chandra pair extension \eqref{extensionHC} we have
the following slight reformulation of Proposition 2.7 in \textit{loc. cit.}.
The vanishing part is a direct 
consequence of Corollary \ref{vanishing_cocycles}.

\begin{prop}
\label{lift}
Let $c \in H^2\left(\mathfrak{g}, \mathbf{K}; V \right)$ be the class of the 
extension \eqref{extensionHC} and $P \to M$ a transitive Harish-Chandra torsor 
over $\left(\mathbf{F}, \mathfrak{g} \right)$. Then $P$ lifts to a transitive Harish-Chandra torsor $\widetilde{P}$ over $\left(\tilde{\mathbf{F}}, \tilde{\mathfrak{g}} \right)$
if an only if the image  $GF\left(c \right) \in H^2_{DR}\left(M; V_P \right)$ is zero. 
If this is the case, the set of equivalence classes of 
such lifts is a torsor over $H^1_{DR}\left(M; V_P \right)$
\end{prop}
In \textit{loc.cit.} this gives a alternative construction for the Deligne 
class of a quantization of the structure sheaf, while we intend to use the statement to 
study obstructions to quantization of bundles. 

\subsection{Torsor point of view for quantizations}
\label{alltheaction}

We continue with the notation of Section 1, assuming we are given 
a deformation quantization $\mathcal{O}_{\hbar}$ of the sheaf 
$\mathcal{O}_M$ of regular functions on an algebraic variety 
compatible with the algebraic symplectic form $\omega$, and 
an order $k$ deformation quantization $E_k$ of a locally free sheaf $E$.

Let $f: Z = Spec\left( S \right) \to U = Spec\left(R \right) \subset M$ be a morphism of 
of schemes with the image in the affine open subset $U$, and 
let $I_f$ be the kernel of the composition 
$$
S \otimes_\kappa \mathcal{O}_{\hbar} \left(U \right) \to 
S \otimes_\kappa R \to S \otimes_\kappa S \to S.
$$
The first map is obtained from reduction modulo $\hbar$ by extension of 
scalars, the second is $Id_S \otimes f^*$ and the last map is the product
on $S$. Let $Jets_{f}^\infty\mathcal{O}_h\left(R\right)$, resp. $Jets_{f}^\infty
E_k \left(R\right)$, be $I_f$-adic completion of $S \otimes_\kappa \mathcal{O}_{\hbar}  \left(U \right)$, resp. $S \otimes_\kappa E_k \left(U \right)$. 
Using non-commutative localization in 
Section 2 of \cite{Yekutieli2013}, one can sheafify both constructions 
in the Zariski topology, obtaining Zariski sheaves $J_f^\infty \mathcal{O}_\hbar,  J_f^\infty E_k$ 
on a scheme $Z$ equipped with a morphism of
schemes $f: Z \to M$. We will write $I, J^\infty \mathcal{O}_\hbar, J^\infty E_k$ 
in the case $Z = M, f = id$.

\begin{theorem}
\label{completion} 
Every point $x$ of $M$ admits an affine open neighborhood $U$ with
 a basis $\left(s_1, \ldots, s_e \right)$ of sections of $E$
trivializing it, and a collection of 1-forms
$\alpha_1, \ldots, \alpha_{2n}$ trivializing $\Omega^1_U$ such that
$$
\omega|_U = \sum_{i =1}^n \alpha_i \wedge \alpha_{i+n}.
$$
Moreover, for such $U$ and a fixed choice of 
$\mathcal{O}_\hbar$, $E_k$ (or $E_\hbar$ if $k = \infty$), one can find 
isomorphisms compatible with the algebra and module structures, respectively:
\begin{equation}
\label{zariski-jets}
J^\infty \mathcal{O}_\hbar|_U \simeq \mathcal{O}_U \widehat{\otimes}_\kappa \mathcal{D}, 
\qquad J^\infty E_k|_U \simeq \mathcal{O}_U \widehat{\otimes}_\kappa 
\mathcal{M}_k
\end{equation}
where $\mathcal{D}$ is the formal Weyl algebra and $\mathcal{M}_k =
\left( \mathcal{D}/\hbar^{k+1} \mathcal{D} \right)^{\oplus e}$. 
Any two such isomorphisms differ by an action of
a regular function with values in
 $\mathbf{F}_k := Aut \left(\mathcal{D}, 
\mathcal{M}_k \right)$.
\end{theorem}

\noindent
\textit{Proof.}
The existence of $s_1, \ldots, s_e$ on a small enough open subset is by the local freeness
of the sheaf of sections of $E$. The rest of the proof is divided into 
several steps

\medskip

\noindent
\textit{Step 1.} (Existence of $\alpha$ forms)

The isomorphism $\widehat{\mathcal{O}_{h}}\cong\mathcal{D}$ follows
from a proposition in section~5.1 of~\cite{BG2022} and we recall the 
argument briefly here. Let $\pi$ be the Poisson bivector associated to 
$\omega$.

First, one chooses a 1-form $\alpha_1$ not vanishing at $x$ (and hence 
in an open neighborhood of $x$). Then, by non-degeneracy of $\pi$, one
can find $\alpha_2$ such that $\langle \pi, \alpha_1 \wedge \alpha_2 
\rangle = f$
is a function that does not vanish at $x$. Shrinking $U$ if necessary 
and replacing $\alpha_2$ by $\frac{1}{f} \alpha_2$ we can achieve 
$\langle \pi, \alpha_1 \wedge \alpha_2 \rangle =1$. Taking
symplectic orthogonal complement of the span of $\alpha_1, \alpha_2$ 
we can similarly find $\alpha_3, \alpha_4$ in that complement, such 
that $\langle \pi, \alpha_3 \wedge \alpha_4 \rangle =1$. Iterating the 
process we get the required symplectic frame of 1-forms. 

\medskip

\noindent
\textit{Step 2.} (Jets of $\mathcal{O}_\hbar$)

Once a symplectic 
basis of 1-forms is chosen, we set $R = \Gamma(U, \mathcal{O})$ and let 
$I_0 = Ker \left(R \otimes R \to R \right)$ be
the ideal of functions vanishing on the diagonal of $U$. We
try to lift $\alpha_1, \ldots, \alpha_{2n}$ from a basis of
$\Omega^1(U) = I_0/I_0^2$ to elements 
$\widehat{\alpha}_1, \ldots, \widehat{\alpha}_{2n}$ in the ideal $\widehat{I}_0$ of
the commutative jets $J^\infty_0 \left(U\right)$ in such a way that, with 
respect to the $R$-linear 
Poisson bracket induced on the jets by $\omega$, we have
$$
\{ \widehat{\alpha}_i, \widehat{\alpha}_{i+n} \} = 1
$$
for $i = 1, \ldots, n$ and other pairwise brackets are zero. 
The proof follows a standard pattern: first we choose 
arbitrary preimages of $\alpha_i$ in $\widehat{I}_0$, and then 
adjust those at step $m \geq 2$ by elements in $\widehat{I}_0^{m+2}$ so that the required Poisson bracket identity
holds modulo $\widehat{I}_0^{m+2}$. At each step one uses the
fact that a relative 2-form of $R[t_1, \ldots, t_{2n}]$ over $R$
is exact iff it is closed. 

Once the Poisson lifts $\widehat{\alpha}_i \in J^\infty_0(U)$
are chosen, one uses a similar procedure to lift them to elements
$t_1, \ldots, t_{2n}$ in 
$J^\infty \mathcal{O}_\hbar (U)$ such that
$$
t_{i+n} t_{i} - t_{i} t_{i+n} = \hbar, \quad \textrm{ for } i = 1, \ldots, n; 
\qquad t_i t_j - t_j t_i = 0 \quad
\textrm{  if  }
\vert i - j \vert \neq n.
$$
In this case, one lifts $\widehat{\alpha}_i$ from 
$\widehat{I}_0$ to $\widehat{I}$ in an arbitrary way and 
then adjusts the lifts by multiples of $\hbar^{m+2}$, $m \geq 2$, to 
ensure that the required commutator identity holds modulo $\hbar^{m+2}$. This process gives a set of generators in 
$J^\infty(U)$ that satisfy the same relations as the generators
of $R \widehat{\otimes} \mathcal{D}$, as required.

\medskip

\noindent
\textit{Step 3.} (Jets of $E_k$)

We proceed by induction on $\ell \geq 0$ showing that a set of 
free generators of $E_{\ell} = E_k/\hbar^{\ell+1} E_k$ over 
$\mathcal{O}_\ell = \mathcal{O}_\hbar/h^{\ell+1} \mathcal{O}_\hbar$
lifts to a set of free generators of $E_{\ell+1}$  over $\mathcal{O}_{\ell+1}$.
 For $\ell = k$
this shows that $E_k$ is locally free over $\mathcal{O}_k$ in the 
Zariski topology. 
Sending $k \to \infty$
we will also obtain the proof for the infinite order quantization
$E_\hbar$. 

In fact, existence of \textit{some} lifted sections follows by induction on $\ell \geq 0$, 
as $H^1\left(U, Ker\left(E_{\ell+1} \to E_\ell \right) \right) = 0$ since 
$E \simeq Ker \left( E_{\ell+1} \to E_\ell \right)$ is coherent 
and $U$ is affine. Let $s^{\ell+1}_1, \ldots, s^{\ell+1}_e$ be a choice of lifted 
sections defining a morphism
$$
\mathcal{O}_{\ell+1}^{\oplus} \to E_{\ell+1}.
$$
We want to show it is necessarily an isomorphism. 

Indeed, to prove surjectivity choose a section $s \in E_{\ell+1}\left(U\right)$ then by inductive
assumption  

$s - \left( a_1 s_1^{\ell+1} + \ldots + a_e s^{\ell+1}_e\right)$ is in 
$\hbar^{\ell+1} E_{\ell+1}$ for some $a_1, \ldots, a_e \in \mathcal{O}_\hbar\left(U\right)$. 
Using the original assumption on the sections $s_1, \ldots, s_e$ of $E = E_0$ we can write
$$
s = \left(a_1 s_1^{\ell+1} + \ldots + a_e s^{\ell+1}_e\right) + \hbar^{\ell+1}\left(b_1 s_1 + \ldots + b_e s_e \right)
$$
for some $b_i \in \mathcal{O}_M\left(U\right)$, as required. 
Injectivity of $
\mathcal{O}_{\ell+1}^{\oplus e} \to E_{\ell+1}
$
is proved similarly: if we had a relation $\sum a_j s^{\ell+1}_j =0$
on an open subset $V \subset U$, with $a_j \in \mathcal{O}_{\ell+1}\left(V\right)$
then dividing out by the largest possible power of $\hbar$ we can
assume that at least one $a_j$ has a non-trivial reduction in 
$\mathcal{O}\left( V \right)$, and after reduction of the relation modulo 
$\hbar$ this contradicts linear independence of $s_1, 
\ldots, s_e$.

\medskip
\noindent
The statement regarding comparison of any two such isomorphisms follows 
from the definitions. 
$\square$

\bigskip
\noindent
Following the outline of Section 6 in \cite{VanDenBergh} we
define a functor
$$
\Psi: Schemes/\kappa \to Set
$$
which sends $Z$ to the set of triples $(f, \varphi, \psi)$
consisting of a morphism $f: Z \to M$ and isomorphisms
$$
\varphi: J^\infty_f \mathcal{O_\hbar}
\simeq \mathcal{O}_Z \widehat{\otimes_\kappa}
\mathcal{D}; \qquad \psi: J^\infty_f E_k 
\simeq \mathcal{O}_Z \widehat{\otimes_\kappa} 
\mathcal{M}_k
$$
compatible with algebra and module
structures, and with filtrations. Theorem 
\ref{completion} shows that locally such isomorphisms indeed exist, and
we can imitate the arguments
in Propositions 6.1.4 and 6.2.2 of \cite{VanDenBergh} to show that $\Psi$ is 
representable by a torsor $P\left(k\right)$ over the proalgebraic 
group  $\mathbf{F}_k := Aut\left(\mathcal{D}, \mathcal{M}_k\right)$. Since
 $$
 \mathfrak{g}_k = Der\left(\mathcal{D} \right) \ltimes \mathfrak{gl}\left(e, \mathcal{D}/
 \hbar^{k+1} \mathcal{D}\right) \supset Lie \left(\mathbf{F}_k \right)
 $$ 
 also acts on $\mathcal{D}$ and 
 $\mathcal{M}_k$, similarly 
 to Section 3.1 of 
 \cite{BezKal03}
 (see discussion before Definition 3.1) the torsor $P(k)$ is upgraded
 to a Harish-Chandra torsor over the pair $\left(\mathbf{F}_k, \mathfrak{g}_k \right)$.

\bigskip
\noindent
For $2k+1 \geq \ell \geq k$, given an order $\ell$ quantization $E_\ell$ that lifts an order 
$k$ quantization $E_\ell$, we have a Harish-Chandra torsor $P\left(\ell\right)$ which 
lifts $P\left(k\right)$ across the extension. By Proposition \ref{lift} 
 existence of $P\left(\ell\right)$ is equivalent to 
the vanishing of a class in $H^2_{DR}\left(M, \mathfrak{h}\left(k, \ell\right)_{P\left(k\right)}\right)$
where the coefficients stand for the flat vector bundle associated to the
Harish-Chandra module $\mathfrak{h}\left(k, \ell\right)$ and the torsor $P\left(k\right)$.
The restriction on $\ell$ is needed to ensure that $\mathfrak{h}\left(k, \ell\right)$ 
is abelian. 
Hence existence of $E_\ell$ implies that the class in de Rham cohomology 
vanishes. Conversely, assume that a choice of $P\left(\ell\right)$ is made.
Apply  the associated bundle construction to the module
$
V = \mathcal{M}_\ell
$
and the torsor $P\left(\ell\right)$ we get a bundle $V_{P\left(\ell\right)}$ with a flat connection 
(since we are working with transitive Harish-Chandra torsors)
and hence we can 
form the de Rham complex $\Omega^\bullet_M\left(V_{P\left(\ell\right)}\right)$. Below
we will prove that $E_\ell$ itself may be recovered as the  
sheaf of flat sections due to the quasi-isomorphism 
$$
E_\ell \mapsto \Omega^\bullet_M\left(V_{P\left(\ell\right)}\right).
$$ 
This gives a way
to recover $E_\ell$ from its Harish-Chandra torsor $P\left(\ell\right)$.
Therefore, the problem of lifting order $k$ quantization $E_k$
to an order $\ell$ quantization $E_{\ell}$ is equivalent to the
problem for lifting $P\left(k\right)$ to $P\left(\ell\right)$. 

For $\ell \geq 2k+2$ we can abelianize the kernel of the
Harish-Chandra extension
and obtain a \textit{necessary} condition for existence of the lift.

\begin{theorem}
\label{jetsandtorsors}
Let $\pi: P \to M$ be a transitive Harish-Chandra torsor
over a Harish-Chandra pair $ \left(\mathbf{F}, \mathfrak{g} \right)$, such that the proalgebraic group $\mathbf{F}$ is a semidirect product $\mathbf{K} \rtimes \mathbf{U}$ of a finite dimensional 
reductive group $\mathbf{K}$ and a prounipotent group $\mathbf{U}$. Let $V$ be a Harish-Chandra
module over the same pair inducing a proalgebraic vector bundle $V_P$
with a flat connection and its de Rham complex 
$\Omega^\bullet_M \left(V_P \right)$.

\bigskip

(i) there exists a quasi-isomorphism 
$$
\Omega^\bullet_M \left(V_P \right) \to \pi_* \left(\Omega_P^\bullet \otimes_\kappa
  V \right)^{\mathbf{K}-basic}
$$

(ii)  if $E_k$ is an order $k$ quantization, $P \left(k \right)$ is 
the associated transitive Harish-Chandra torsor over 
$ \left( \mathbf{F}, \mathfrak{g} \right) = \left(Aut \left(\mathcal{D}, \mathcal{M}_k \right), \mathfrak{g}_k \right)$ 
then the vector bundle $V_{P \left(k \right)}$  associated to 
$V = \mathcal{M}_k$ is isomorphic to
$J^\infty E_k$. There exists a canonical quasi-isomorphism 
$$
E_k \to \Omega^\bullet_M \left( V_{P \left(k \right)}\right)
\simeq \Omega^\bullet_M(J^\infty E_k).
$$
Similarly, the sheaf of flat sections of the bundle associated to 
$V = gl_e \left(\mathcal{D}/\hbar^{k+1} \mathcal{D} \right)$ is isomorphic to the 
sheaf of quantized isomorphisms $End_{\mathcal{O}_h}  \left(E_k \right)$.

(iii)
if $P \left(k \right)$ is a transitive Harish-Chandra torsor over 
$ \left(Aut \left(\mathcal{D}, V \right), \mathfrak{g}_k \right)$ lifting the transitive
Harish-Chandra torsor $P \left(0\right)$ corresponding to $E$ then the
sheaf $E_k$ of flat sections of  $V_{P \left(k \right)}$ is an order $k$ quantization
of $E$. Moreover, the transitive Harish-Chandra torsor associated to $E_k$ is 
canonically isomorphic to $P \left(k \right)$.
\end{theorem}

\bigskip
\noindent 
\textbf{Proof of Theorem \ref{jetsandtorsors}:} For part (i), 
consider the quotient $Q = P/\mathbf{U}$
which is a $\mathbf{K}$-torsor and let
 $\rho: Q \to M$ be the projection. Let $\mu: P \to Q$ be the quotient map
 and $\mathcal{V}_P$ the bundle on $Q$ associated to the $\mathbf{U}$-torsor
 $\mu: P \to Q$. The $\mathcal{V}_P$ has a $\mathbf{K}$-equivariant 
 structure and its further equivariant descent to $M$ is exactly 
 $V_P$. Furthermore, the transitive Harish-Chandra structure on 
 $P$ induces a flat structure on $\mathcal{V}_P$ (as we can think of $\mu: P \to Q$
 as a transitive Harish-Chandra torsor over the pair $(\mathbf{U}, 
 \mathfrak{g})$). 
Since $\mathbf{U}$ is pro-unipotent, the embedding 
$$
\Omega_Q^\bullet (\mathcal{V}_P) \to \mu_* (\Omega^\bullet_P 
\otimes_\kappa V)
$$
is a quasi-isomorphism by the proalgebraic version of Poincare
lemma, cf. \cite{VanDenBergh}, since $P$ is trivial 
over small enough affine open sets $U$ by Theorem \ref{completion}
(so that $P|_U \simeq Q|_U \times \mathbf{U}$). 
Pushing this forward by $\rho$
and taking the $\mathbf{K}$-basic forms, gives 
(i).

For part (ii): by definition of $P \left( k \right)$, the pullback of $E_k$ to its total space
has a natural morphism to the trivial bundle with fiber $V
= \mathcal{M}_k$. Hence we 
obtain a morphism on $M$ from $E_k$ to the associated bundle $V_{P(k)}$. 
To show that this morphism induces a quasi-isomorphism between 
$E_k$ and the de Rham complex of $V_{P\left(k \right)}$ it suffices to restrict to an 
affine open subset where the isomorphisms \eqref{zariski-jets} hold and hence
$P(k)$ is trivialized. Since $E_k$ splits into direct sum of several copies of $\mathcal{O}_\hbar/
\hbar^{k+1} \mathcal{O}_\hbar$, it suffices to assume that $E_k = 
\mathcal{O}_\hbar/
\hbar^{k+1} \mathcal{O}_\hbar$. Moreover, all morphisms are compatible with 
 $\hbar$-adic filtrations and hence by a standard spectral sequence argument 
 it suffices to show that quasi-isomorphisms hold modulo $\hbar$. But then 
 the reduction of 
 $V \left(k \right)_{P \left(k \right)}$ modulo $\hbar$ is isomorphic to the jets of the 
 structure sheaf $\mathcal{O}_M$  and the result follows from Theorem 4.4 
 in \cite{Ye}.
The same argument applies to $End_{\mathcal{O}_\hbar} \left(E_k \right)$ and the
bundle associated to the module $gl_e \left(\mathcal{D}/\hbar^{k+1} \mathcal{D} \right)$.

For (iii), observe that the action of 
$\mathcal{D}$ on $V = \mathcal{M}_k$ 
is compatible with the Harish-Chandra
module structures. This induces an action of the bundle associated to 
$\mathcal{D}$ on the bundle associated to $V$, and this action is compatible
with the connections on both bundles. Hence the sheaf of flat sections of the 
former (which is $\mathcal{O}_\hbar$ by Section 5 of \cite{bib:TN}) acts on the sheaf of flat sections $E_k$
and this exhibits $E_k$ as an order $k$ quantization of $E$.
Indeed, the action of $\mathcal{O}_\hbar$ descends to 
action modulo $\hbar^{k+1}$, and since the torsor $P \left(k \right)$ is 
Zariski locally trivial, the sheaf $E_k$ of flat sections
of $V(k)$ is
Zariski locally isomorphic to $\left(\mathcal{O}_\hbar/
\hbar^{k+1} \mathcal{O}_\hbar \right)^{\oplus e}$. The same
observation also shows that the 
product map 
$$
\mathcal{D}_{P\left(k \right)} \otimes_{\mathcal{O}_\hbar} E_k 
\to V\left(k \right)_{P\left(k \right)}
$$
is an isomorphism. 

Pulling back the natural embedding $E_k \to V\left(k \right)_{P\left(k \right)}$ 
with respect to  $\pi: P\left(k \right) \to M$ we see that 
$\pi^* E_k$ is the space of flat sections of
$\mathcal{O}_{P\left(k \right)} \otimes_\kappa V\left(k \right)$ with respect to the
connection induced by the Harish-Chandra structure. 
Now we prove that 
$$
J^\infty \mathcal{O}_\hbar \otimes_{\mathcal{O}_\hbar} E_k
\simeq J^\infty E_k
$$
Indeed on $M \times M$ we have a map
$$
\left(\mathcal{O}_M \otimes_\kappa \mathcal{O}_\hbar \right)/I^r 
\otimes_{\mathcal{O}_\hbar}  E_k 
\to \left(\mathcal{O}_M \otimes_\kappa E_k \right)/I^r \left(\mathcal{O}_M \otimes_\kappa E_k \right)
$$
which is an isomorphism due to right exactness of tensor
product, and taking the limit $r \to \infty$ we get
the isomorphism for jets. 

Since the quantization 
$\mathcal{O}_\hbar$ was chosen and fixed throughout our argument
we also have $J^\infty \mathcal{O}_\hbar \simeq 
\mathcal{D}_{P\left(k \right)}$ which finally gives an isomorphism 
$$
J^\infty E_k \simeq V\left(k \right)_{P\left(k \right)}.
$$
Finally, denote temporarily by $P'\left(k \right)$ the torsor associated to 
jet trivializations of $E_k$ as in \eqref{zariski-jets}.
Pulling back the above isomorphism to $P\left(k \right)$ we get
an induced morphism $P\left(k \right) \to P'\left(k \right)$ which is 
$\mathbf{F}$-equivariant and hence is an isomorphism of
torsors. 
$\square$

\bigskip
\noindent
\textbf{Proof of Theorem 1}.

If $E_\ell$ exists then Harish-Chandra 
torsor $P\left(k \right)$ lifts to the Harish-Chandra torsor $P\left(\ell\right)$ and the
obstruction class vanishes by Proposition \ref{lift}. Conversely, 
if the obstruction class vanishes
then the lift $P\left(\ell\right)$ exists.
By Theorem \ref{jetsandtorsors} (iii) it induces 
an order $\ell$ quantization $E_\ell$. By the same result, equivalence
classes of  quantizations are in bijective correspondence with 
equivalence classes of torsor lifts.

The latter by Proposition \ref{lift} 
are parametrized by the first cohomology $H^1\left(M, End\left(E\right)_{\left(k, \ell\right]}\right)$. $\square$

\bigskip
\noindent
\textbf{Definition of $End^+_{\mathcal{O}_\hbar}\left(E_k\right)$}.
In view of part (ii) of the previous theorem, we define
$End^+_{\mathcal{O}_\hbar}\left(E_k\right)$ as the sheaf of flat sections
of the bundle associated to the Harish-Chandra torsor $P\left(k \right)$
and the Harish-Chandra module $\mathfrak{h}^+\left(k\right) = \mathfrak{h}^+\left(k, 2k+2\right)$
over the pair $\left(\mathbf{F}_k, \mathfrak{g}_k\right)$, defined in 
\eqref{eq:Lie_alg_SES-prime}. 

Observe that there is a short exact sequence of modules 
$$
0 \to \hbar^{2k+2} \cdot \mathcal{A} \to \mathfrak{h}^+\left(k, 2k+2\right)
\to \mathfrak{h}\left(k, 2k+1\right) \to 0
$$
which is not split on the level of $\mathfrak{g}_k$-modules since 
the image of $\hbar^{2k+1} \mathcal{A} \cdot Id \subset \mathfrak{h}\left(k, 2k+1\right)$ in $\mathfrak{h}^+\left(k, 2k+2\right)$ 
would not be stable under the action of $(\mathcal{D}/\hbar^{k+1} \mathcal{D})Id
\subset gl_e(\mathcal{D}/\hbar^{k+1} \mathcal{D}) \subset 
\mathfrak{g}_k$ due to the nontriviality of the Poisson bracket on 
$\mathcal{A}$.

\bigskip
\noindent
\textbf{Proof of Theorem 2}. If $E_\ell$ exists, consider 
the Harish-Chandra torsor $P\left(\ell\right)$ associated to it. 
Let $H'\left(\ell\right)  \subset Aut\left(\mathcal{D}, \mathcal{M}_\ell\right) $ be the subgroup
of automorphisms which are identity on $\mathcal{D}$ and 
on $\mathcal{M}_\ell$ have the form $exp \left(\hbar^{2k+1} X\right)$ with 
$X \left(mod\ \hbar\right)$ having trace zero. 

Then $\mathbf{F}\left(\ell\right)/H_\ell'$ is the group part of the 
abelian Harish-Chandra extension with the Lie part 
\eqref{eq:Lie_alg_SES-prime}, hence $P\left(\ell\right)/H_\ell'$ is a 
torsor lift of $P\left(k\right)$ and its existence implies by Proposition 
\ref{lift} that the obstruction class vanishes. $\square$

\section{The holomorphic case and Fedosov connections for bundles.}

\subsection{Existence of a Dolbeault model}

In this section we focus on the case $\kappa = \mathbb{C}$, 
working in the setting of complex manifolds and analytic 
topology. Following [BNT], observed that for there are $C^\infty$
isomorphisms: 
\begin{equation}
\label{Cinfty}
\mathcal{A}^0 \left(J^\infty \left(\mathcal{O}_h \right) \right)
\simeq \mathcal{A}^0 \left(S_\hbar \left(\Omega_M \right) \right); 
\end{equation}
where $\mathcal{A}^0$ stands for the sheaf of smooth sections
of a holomorphic bundle (in the analytic topology) and 
$S_\hbar$ is the completed symmetric algebra of the 
holomorphic cotangent
bundle $\Omega_M$ with coefficients in $\mathbb{C} \left[ \left[\hbar \right] \right]$.
Moreover, this isomorphism is compatible with multiplicative 
structures on both sides (but not with the natural 
Dolbeault differentials).
We claim that there is a similar isomorphism 
\begin{equation}
\label{smoothE}
s_k: \mathcal{A}^0 \left(J^\infty \left(E_k\right)\right) \simeq 
\mathcal{A}^0 \left( \left[S_\hbar \left( \Omega_M \right)/ \left( \hbar^{k+1}\right) \right] 
\otimes_{\mathcal{O}_M} E \right);
\end{equation}
which agrees with the module structures on both sides (not with the differentials).

Indeed, by the end of the previous section and Section 4 of [BK]
the jet bundles are associated with the torsor $P \left( k \right)$
and $(\mathbf{F}_k, \mathfrak{g}_k)$-modules 
$\mathcal{D}$ and $V\left( k \right)$, 
respectively. Both carry filtrations induced by the powers of
the maximal ideal $\mathfrak{m} \subset \mathcal{D}$ and the 
completed associated graded objects are isomorphic to the objects
on the right. Starting with the natural isomorphisms modulo $\hbar$ we inductively lift them to $C^\infty$ isomorphisms 
modulo $\hbar^{\ell+1}$ for 
$\ell \geq 0$. Moreover, these isomorphisms can be chosen  
to be compatible with the algebra and module structures. 
At each step obstructions to such a lift
are measured by cohomology classes of $C^\infty$ sections 
of vector bundles. These cohomology groups vanish due to existence
of partition of unity for such sections.

\bigskip

From now on, for a holomorphic bundle 
$W$ we write $\mathcal{A}^j(W)$ for $C^\infty$ differential 
forms of degree $j$ with values in $W$. 
The sum of the holomorphic connection (induced by the Harshi-Chandra structure) and 
the Dolbeault operator of the jet bundle, acting on
$\mathcal{A}^\bullet \left(J^\infty \left(E_k \right) \right)$, 
becomes just a smooth connection on the $[S_\hbar(\Omega_M)/(\hbar^{k+1})] \otimes_{\mathcal{O}_M} E$ side, which 
is compatible with the action of $S_\hbar(\Omega_M)$ since 
the isomorphisms \eqref{smoothE} are compatible with the module
structures. Let $\nabla_\hbar$ be the corresponding connection on 
$S_\hbar \left(\Omega_M \right)$, i.e. the Fedosov connection for the quantized
functions \cite{bib:TN}. Then the connection on $[S_\hbar(\Omega_M)/(\hbar^{k+1})] \otimes_{\mathcal{O}_M} E$
has the form 
$$
\nabla^{tot}_k = \nabla_\hbar \otimes Id_E + Id_{S_\hbar \left(\Omega_M \right)} 
\otimes\nabla^E_k
$$
with unknown 
$$
\nabla^E_k = \nabla^E + \overline{\partial} + \hbar C_1 + \hbar^2 C_2 + \ldots + \hbar^k C_k: 
\mathcal{A}^0 \left(E \right) \to \mathcal{A}^1 \left([S_\hbar(\Omega_M)/(\hbar^{k+1})] \otimes_{\mathcal{O}_M} E \right).
$$
where $\nabla^E: \mathcal{A}^0 \left(E \right) \to \mathcal{A}^1 \left(S \left(\Omega_M \right) \otimes_{\mathcal{O}_M} E \right)$ is obtained from the Grothendieck connection on 
$J^\infty \left(E_0 \right)$ and each $C_i$ is an $\mathcal{A}^0 \left(\mathcal{O}_M \right)$-linear 
operator obtained from a section of $\mathcal{A}^1 \left(\hbar^iS(\Omega_M) \otimes End \left(E \right) \right)$. 

As  $E_k$ is the sheaf of flat sections of $\nabla^{tot}_k$,
describing the order $k$ quantization $E_k$ is equivalent to 
describing the operator 
$
\nabla^E_k
$
such that the induced operator $\nabla^{tot}_{k}$ is flat modulo
$\hbar^{k+1}$. Extending an order $k$ quantization to an 
order $\ell$ quantization is equivalent to finding the second part of the
expression 
$$
\left[ \nabla^E + \overline{\partial} + \hbar C_1 + \ldots + \hbar^k C_k \right] + \hbar^{k+1}
\left[C_{k+1}+ \hbar{C_{k+2}} + \ldots + \hbar^{\ell-k-1} C_{\ell} \right]
$$
in such a way that the zero curvature condition is observed 
modulo $\hbar^{\ell+1}$.
Note that under the assumption $\ell \leq 2k+1$ we will have $\ell - k - 1 \leq k$. We will show that the term 
 $\nabla^E_{ \left(k, \ell \right]}$ given by the second pair of 
brackets defines a cohomology class of the quantized endomorphisms as
in Theorem 1. The next lemma is essentially a repetition of Theorem 1 in the context of connections.

\begin{lemma}
For $k+1 \leq \ell \leq 2k+1$ the complex
$$
\left(\bigoplus_{r=0}^{\ell-k-1} 
\hbar^r \mathcal{A}^\bullet \left( S \left( \Omega^M \right) \otimes End \left( E \right) \right), \left[ \nabla^{tot}_k, \cdot \right] \right)
$$
computes cohomology of the sheaf 
$End_{\mathcal{O}_\hbar} \left( E_k \right) / \hbar^{\ell-k}End_{\mathcal{O}_\hbar} \left( E_k \right)$.
The element 
$
\frac{1}{\hbar^{k+1}}\left(\nabla^{tot}_k\right)^2 
$
represents a second cohomology class in the complex. It vanishes if and 
only if $E_k$ admits an extension to an order $\ell$ deformation $E_\ell$, in 
which case the set of isomorphism classes of such $E_\ell$ is in 
bijective correspondence with the first cohomology group of the above complex.
\end{lemma}
\textit{Proof.} We have fixed an $C^\infty$ isomorphism of bundles
$J^\infty(E_k)$ and $[S_\hbar(\Omega_M)/(\hbar^{k+1})] \otimes_{\mathcal{O}_M} E$
which is multiplicatively compatible with the similar $C^\infty$ isomorphism 
of $J^\infty(\mathcal{O}_\hbar)$ and $S_\hbar(\Omega_M)$ on functions. By the previous section, 
both $E_k$ and
$\mathcal{O}_\hbar$ are quasi-isomorphic to the holomorphic de Rham complexes constructed 
from the action of Grothendieck connections on their jets of functions. The same holds for
the endomorphism bundle $End_{\mathcal{O}_\hbar} (E_k)$ (where the Grothendieck connection 
on jets of endomorphisms of $E_k$ is induced by the commutator with the Grothendieck 
connection on jets of $E_k$ itself). It follows that 
$$
\left(\bigoplus_{r=0}^{k} 
\hbar^r \mathcal{A}^\bullet \left( S \left( \Omega^M \right) \otimes End \left( E \right) \right), \left[ \nabla^{tot}_k, \cdot \right] \right)
$$
is a fine resolution of $End_{\mathcal{O}_\hbar} (E_k)$. The terms with $r = \ell-k, \ldots, k$
form a resolution of $\hbar^{\ell-k} End_{\mathcal{O}_\hbar} (E_k)$, hence the 
sum of terms with $r = 0, \ldots, \ell - r -1$ resolves 
$End_{\mathcal{O}_\hbar} \left( E_k \right) / \hbar^{\ell-k}End_{\mathcal{O}_\hbar} \left( E_k \right)$.

\bigskip
\noindent
By the flatness assumption for 
$\nabla^{tot}_k$ modulo $\hbar^{k+1}$ its full curvature element
modulo $\hbar^{\ell+1}$ has the form:
$$
R_{k, \ell} := \left(\nabla^{tot}_k \right)^2 \in \hbar^{k+1} \left(\bigoplus_{r=0}^{\ell-k-1} 
\hbar^r \mathcal{A}^2 \left( S_\hbar \left( \Omega^M \right) \otimes End \left( E \right) \right)\right)
$$
Thus the curvature condition imposed on $\nabla^{tot}_k + \hbar^{k+1} \nabla^E_{ \left( k, \ell \right]}$
reads 
$$
R_{k, \ell} + \hbar^{k+1} \left[ \nabla^{tot}_k, \nabla^E_{\left(k, \ell \right]} \right] = 0 \in 
\hbar^{k+1} \left(\bigoplus_{r=0}^{\ell-k-1} 
\hbar^r \mathcal{A}^2 \left( S_\hbar \left(\Omega^M \right) \otimes End \left( E \right) \right)\right)
$$
where we have used $2 \left( k+1 \right) > \ell$, so the equation is affine linear in the unknown
$\nabla^E_{ \left( k, \ell \right] }$.
Therefore, the existence of the unknown element 
$$
\nabla^E_{\left( k, \ell \right] } \in \bigoplus_{r=0}^{\ell-k-1} 
\hbar^r \mathcal{A}^1 \left( S \left( \Omega^M \right) \otimes End\left( E \right) \right)
$$
is equivalent to the vanishing of the 
second cohomology 
class of $\frac{1}{\hbar^{k+1}}R_{k, \ell}$. Note that this element is indeed
closed by the Bianchi Identity for $\nabla^{tot}_k$
(viewed modulo $\hbar^{\ell+1}$).

 The vanishing
of the corresponding cohomology class implies that
$\nabla^{E}_{ \left( k, \ell \right]}$ exists, and two such elements that 
differ by $\left[ \nabla^{tot}_k, \alpha \right]$ (where $\alpha$ is
a degree 0 form) correspond 
to two lifts of \eqref{smoothE} from order $k$ to 
order $\ell$, that differ by $Id + \hbar^{k+1}\alpha$. Thus, the action of  
$\left[ \nabla^{tot}_k, \alpha \right]$ identifies 
the two sheaves of flat sections. Conversely, and isomorphism of two order $\ell$
quantizations which reduces to identity modulo $\hbar^{k+1}$ will induce
an automorphism on the right hand side of \eqref{smoothE} which is equal to 
identity modulo $\hbar^{k=1}$. Subtracting $Id$ we will get the element
$\hbar^{k+1} \alpha$ and its coboundary with respect to $[\nabla_k^{tot}, \cdot]$ will 
be the difference between the two choices of $\nabla^E_{(k, \ell]}$ leading to 
equivalent order $\ell$ quantizations. 
$\square$

\subsection{The case $\ell \geq 2k+2$.}

For $\ell \geq 2k+2$ we only have a necessary condition (a sufficient
would involve nonabelian cohomology) and the obstruction sheaf 
$End_{\mathcal{O}_\hbar}^+ \left( E_k \right)$ is the same for all $\ell$ in 
the required range. In particular, we can just assume $\ell = 2k+2$.
The sheaf of interest can be resolved by the complex of sheaves
\begin{equation}
\label{end-resolution}
\left[\mathcal{A}^\bullet \left(S_\hbar \left( \Omega_M \right) \otimes End\left(E\right)/ \hbar^{k+1} S_\hbar \left( \Omega_M \right) \otimes 
End^0\left(E\right) + 
\hbar^{k+2} S_\hbar \left( \Omega_M \right) \otimes 
End \left(E\right) \right), \left[ \nabla^{tot}_k, \cdot \right] \right]
\end{equation}
where $End^0(E)$ is the sheaf of trace zero endomorphisms over $\mathcal{O}_M$.
We note here that $[\nabla^{tot}_k, \cdot]$ is indeed a differntial:
by assumption $\left( \nabla^{tot}_k\right)^2$ is a multiple of 
$\hbar^{k+1}$ hence the commutator with this operator vanishes on 
the above quotient. 

\begin{theorem}
The above complex \eqref{end-resolution} is quasi-isomorphic to $End^+_{\mathcal{O}_\hbar} \left(E\right)$
as defined at the end of Section 4. The expression $\frac{1}{\hbar^{k}} \left(\nabla_k^{tot}\right)^2$
projects to a closed section in homological degree 2. Its 
class in $H^2\left(M, End^+_{\mathcal{O}_\hbar} \left(E\right)\right)$ vanishes if 
$E_k$ admits a lift to an order $\ell$ quantization for $\ell \geq 2k+2$.
\end{theorem}

\subsection{Comparison with Gelfand-Fuks map.}

The definitions below follow the exposition in Section 2.2
of \cite{GdKN}. 

Let $\rho: R \to M$ be torsor over $Sp\left(2n, \mathbb{C}\right) \times 
GL(e, \mathbb{C})$ of pairs consisting of a (local) symplectic
frame of $T_M$ and a (local) frame of $E$. It is clear 
that the projection $\pi_k: P_k \to M$ factors through $\rho$, indeed
the group part
 $\mathbf{F}_k = Aut\left(\mathcal{D}, \mathcal{M}_k\right)$
 of the Harish-Chadra pair, admits a semidirect product decomposition 
 $\mathbf{K} \ltimes \mathbf{U}_k$ with prounipotent $\mathbf{U}_k$
and reductive 
$\mathbf{K} \simeq Sp\left(2n, \mathbb{C}\right) \times GL\left(e, \mathbb{C}\right)$ 
and in these terms $R = P_k/U_k$ with the induced $\mathbf{K}$-action.

Note that the choices of the two $C^\infty$-isomorphisms in the 
beginning of this section
define a $\mathbf{K}$-equivariant smooth section $s_k: R \to P_k$.
As $P_k$ has transitive Harish-Chandra structure, its tangent bundle 
is trivialized and the corresponding $\mathfrak{g}_k$-valued 1-form
$\omega_k \in \Gamma \left(P_k; \Omega^1 \otimes_\mathbb{C} \mathfrak{g}_k\right)$
satisfies the Maurer-Cartan equation $d \omega_k + \frac{1}{2} [\omega_k, 
\omega_k] = 0$. The same holds for its pullback $S_k := s_k^*(\omega_k) 
\in \Gamma \left(R, \mathcal{A}^1 \left(\mathcal{O}_R \otimes_{\mathbb{C}} \mathfrak{g}_k\right)\right)$, which gives the trivial 
bundle on $R$ with fiber $\mathfrak{g}_k$ a
non-trivial flat connection $\nabla_k = d + S_k$.
For any $\ell \geq k+1$ consider the following 
version of the Gelfand-Fuks map

\begin{equation}
\begin{array}{c}
GF:C_{Lie}^{\bullet}\left(\mathfrak{g}_{k}, 
\mathbf{K};\mathfrak{h}\left(k, \ell\right)\right)\rightarrow
\Gamma\left(R; 
\mathcal{A}^\bullet \left(\mathcal{O}_R \otimes_\mathbb{C}\mathfrak{h}\left(k, \ell\right)\right)\right)^{\mathbf{K}-basic}\\
GF\left(\eta\right)\left(X_{1},\cdots,X_{j}\right)=\eta\left(S_k\left(X_{1}\right),\cdots,S_k\left(X_{j}\right)\right)
\end{array}\label{eq:GF_map_defn_B}
\end{equation}
$\mathbf{K}$-invariance of the Lie cocycle $\eta$ 
ensures that $GF(\eta)$ is a closed element in the subspace of $\mathbf{K}$-basic forms.  
Note that this space may be identified with 
forms on $R/\mathbf{K} = M$ with values in the associated bundle
$$
\mathfrak{h}\left(k, \ell\right)_{R}
\simeq \bigoplus_{r = k+1}^\ell \hbar^k S\left(\Omega_M\right) \otimes 
End\left(E\right), 
$$
which has a flat connection induced by $\nabla_k$.
In particular, we can apply the Gelfand-Fuks map to the extension 2-cocycle
$$
C \in C^2 \left( \mathfrak{g}_k, \mathbf{K}; \mathfrak{h} \left(k, \ell \right) \right).
$$
associated with a vector space  splitting
$\mathfrak{g}_\ell = \mathfrak{g}_k \oplus \mathfrak{h}\left(k, \ell\right)$.
A more direct way of seeing that the image of $C$ must give a 
zero class in cohomology if $E_\ell$ exists (for $\ell 
\in \{k+1, \ldots, 2k+1\}$) is to observe
that in this case
the $\mathfrak{g}_k$-valued form $S_k$ on $R$ extends to 
a $\mathfrak{g}_\ell$-vallued form $S_\ell$ and then compare
the two Maurer-Cartan equations. Conversely, if 
$S_k$ extends to $S_\ell$ then the connection 
$d + S_\ell$ induces a $\mathbf{K}$-invariant connection 
on $\mathcal{O}_R \otimes_\mathbb{C}\mathcal{M}_\ell$. Both the bundle and the connection 
descend to $M$ and we can set $E_\ell$ to be the kernel of the
descended connection. 

\section{Special Cases and Applications.}

Observe that for $\ell = k+1$ the cohomology 
groups of Theorem 1 are simply the cohomology groups 
$H^1\left(M, End\left(E\right)\right)$ and $H^2\left(M, End\left(E\right)\right)$ of Zariski 
sheaves of $\mathcal{O}_M$-modules. 
For $\ell \in \{k+1, 2k+1\}$ the sheaf
$End\left(E\right)_{\left( k, \ell \right]}$ admits a filtration by powers of 
$\hbar$ and the associated graded sheaf of this filtration 
is isomorphic to the direct sum of $\left(\ell - k\right)$ copies of 
$End_{\mathcal{O}_M}\left(E\right)$. 
Thus, by general theory, there is a spectral sequence with 
$$
E_1^{p, q} = H^{p+q} \left(M, \hbar^{k+1+p} \cdot End\left(E\right)\right), \qquad p = 0, \ldots, 
\left(\ell-k-1\right), p+q = 0, \ldots, 2n
$$
which converges to $H^{p+q} \left(M, End\left(E\right)_{\left(k, \ell\right]}\right)$. 
In particular, vanishing of $H^i\left(M, End\left(E\right)\right)$ implies  vanishing of $H^i\left(M, End\left(E\right)_{\left(k, \ell\right]}\right)$. Either using the spectral sequence 
argument or using the case    $\ell = k+1$ repeatedly, we obtain 
\begin{cor}
If $H^2\left(M, End\left(E\right)\right) = 0$, e.g. if $M$ is a union of two 
affine open sets, then any order $k$ deformation quantization 
$E_k$ extends to an order $\ell$ deformation quantization for 
all    $\ell \geq k+1$. If, in addition $H^1\left(M, End\left(E\right)\right) = 0$, e.g.
if $M$ is affine, such an extension is unique. In particular,
the restriction of $E_k$ to any affine open subset in $M$ admits
a unique extension to quantization of infinite order. 
\end{cor}

\begin{cor}
If $M$ is affine and $E$ is the trivial bundle then 
any quantization of order $k \in \left[ 1, \infty \right)$ is necessarily trivial. 
\end{cor}

\medskip
\noindent
\textbf{Remark.} It would be interesting to 
compute the differentials of the spectral sequence in 
terms of invariants encoding the choice of $\mathcal{O}_\hbar$
and $E_k$. Theoretically, this could allow one to find situations when 
$H^2 \left( M, End\left(E\right) \right)$ does not vanish, but $H^2\left(E, End_{\mathcal{O}_\hbar}\left(E_k\right)_{\left(k, \ell\right]}\right)$ does. 

\section{Lifts of Morphisms}

In this section, we assume that $k+1 \leq \ell \leq 2k+1$ and
set $s = l-k-1$. Then $0 \leq s \leq k$. Assume 
that we are given an order $\ell$ deformation quantization 
$E_l$, and we set $E_k = E_\ell/\hbar^{k+1} E_\ell$, 
$E_s = E_\ell/\hbar^{s+1} E_\ell$. We can view $E_\ell$
as an order $\ell$ extension of $E_k$, which also fits
a short exact sequence 
$$
0 \to \hbar^{k+1} \cdot E_s \to E_\ell \to E_k \to 0
$$
of sheaves of $\mathcal{O}_\hbar$-modules. As such, it 
corresponds to a class $e_\ell \in Ext^1_{\mathcal{O}_{\hbar}}
(E_k, \hbar^{k+1} \cdot E_s)$. 

Suppose that $F_\ell, F_k, F_s$ are similar deformation 
quantizations of a locally free sheaf $F$, corresponding to 
a similar class $f_\ell \in Ext^1_{\mathcal{O}_{\hbar}}
(F_k, \hbar^{k+1} \cdot F_s)$.

Given a morphism $\varphi_k: E_k \to F_k$ of sheaves of
$\mathcal{O}_\hbar$-modules, we would like to know 
whether it can be lifted to a morphism $\varphi_\ell: E_\ell
\to F_\ell$. Note that as $s \leq k$, the reduction of 
$\varphi_k$ mod $\hbar^{s+1}$ gives a morphism $\varphi_s:
E_s \to F_s$.
So we would like to have a morphism $\varphi_\ell$ which
restricts to $\hbar^{k+1} \cdot \varphi_s$ on $\hbar^{k+1} \cdot E_s$
and reduces to $\varphi_k$ modulo $\hbar^{k+1}$. 
By the Baer Criterion, cf. Section 3.4 of \cite{Weibel}, existence
of such morphism 
of extensions is equivalent to 
\begin{equation}
\label{morphism_obstruction}
O : = (\hbar^{k+1} \cdot \varphi_s) \circ e_\ell - f_\ell \circ 
\varphi_k = 0 \in Ext^1_{\mathcal{O}_\hbar}(E_k, \hbar^{k+1} \cdot F_s).
\end{equation}
where $\circ$ stands for the Yoneda pairing. Indeed, $E_\ell$
is quasi-isomorphic to the derived category 
object $E_k \to \hbar^{k+1} \cdot E_s[1]$
(with the arrow induced by the extension class),
and the vanishing of the obstruction ensures that there is 
a morphism into a object $F_k \to \hbar^{k+1}\cdot F_s[1]$.
The resulting morphism on cohomology gives the required
$E_\ell \to F_\ell$. 

What if $O \neq 0$? In that case we could adjust our choices
of $E_\ell, F_\ell$, replacing them with new lifts 
$E_\ell', F_\ell'$ of the same order $k$ quantizations
$E_k, F_k$. 

\begin{prop}
The obstruction class $O$ belongs to a canonical subgroup of 
$Ext^1_{\mathcal{O}_\hbar}(E_k, \hbar^{k+1} \cdot F_s)$
isomorphic to $H^1(M, 
\mathcal{H}om_{\mathcal{O}_\hbar}(E_s, F_s))$.
Suppose that for $a \in H^1(M, End_{\mathcal{O}_\hbar} (E_s)), 
b \in H^1(M, End_{\mathcal{O}_\hbar} (F_s))$ we have
$$
O = b \circ \varphi_k - (\varphi_s) \circ a
$$
in $H^1(M, \mathcal{H}om_{\mathcal{O}_\hbar}(E_s, F_s))$. Then there exist
order $\ell$  quantizations $E'_\ell$, $F'_\ell$ extending 
$E_k, F_k$, respectively, and a morphism $\varphi_\ell: 
E'_\ell \to F'_\ell$ lifting $\varphi_k$. In this case $\varphi_\ell$
is determined uniquely up to an element of 
$H^0(M, 
\mathcal{H}om_{\mathcal{O}_\hbar}(E_s, F_s)) = Hom_{\mathcal{O}_\hbar}
(E_s, F_s)$.
\end{prop}
\textbf{Proof.} First consider the standard local-to-global spectral 
sequence for Ext groups with the second term
$$
E_2^{pq} = H^p(M, \mathcal{E}xt^q_{\mathcal{O}_\hbar}(E_k, F_s)) \Rightarrow 
Ext^{p+q}_{\mathcal{O}_\hbar}(E_k, F_s) 
$$
Next, observe that by change of rings spectral sequence and the 
fact that $E_k$ is induced via the morphism $\mathcal{O}_\hbar \to 
\mathcal{O}_k$, we have another spectral sequence with 
$$
E_2^{pq} = \mathcal{E}xt^{p}_{\mathcal{O}_k}(E_k, 
\mathcal{E}xt^q_{\mathcal{O}_\hbar} (\mathcal{O}_k, F_s))
\Rightarrow \mathcal{E}xt^{p+q}_{\mathcal{O}_\hbar}(E_k, F_s)
$$
As $E_k$ is projective over $\mathcal{O}_k$, in the latter spectral 
sequence only the terms with $p=0$ are nonzero. Further, since 
$\mathcal{O}_k$ has an $\mathcal{O}_\hbar$-projective resolution $\mathcal{O}_\hbar \xrightarrow{\hbar^{k+1}} 
\mathcal{O}_\hbar$, the nonzero terms only occur for $q=0, 1$ and in this case
$$
\mathcal{E}xt^0_{\mathcal{O}_\hbar} (\mathcal{O}_k, F_s) \simeq
\mathcal{E}xt^1_{\mathcal{O}_\hbar} (\mathcal{O}_k, F_s)
\simeq F_s
$$
as $s \leq k$ and thus $\hbar^{k+1}$ acts by zero on $F_s$. Note also 
that 
$$
\mathcal{H}om_{\mathcal{O}_k} (E_k, F_s) \simeq 
\mathcal{H}om_{\mathcal{O}_\hbar} (E_k, F_s) \simeq 
\mathcal{H}om_{\mathcal{O}_\hbar} (E_s, F_s)
$$
as $s \leq k$. So the former local-to-global spectral sequence has two nontrivial 
rows, for $q = 0, 1$ and both are given by $H^p(M, \mathcal{H}om_{\mathcal{O}_\hbar} (E_s, F_s))$. 
In particular, only the $E_2$ term can have nonzero differentials and 
the spectral sequence information essentially reduces to the 
long exact sequence
$$
0 \to H^1(M, \mathcal{H}om_{\mathcal{O}_\hbar} (E_s, F_s)) \to 
Ext^1_{\mathcal{O}_\hbar} (E_k, F_s) \to 
H^0(M, \mathcal{H}om_{\mathcal{O}_\hbar} (E_s, F_s)) \to 
H^2(M, \mathcal{H}om_{\mathcal{O}_\hbar} (E_s, F_s)) \to \ldots
$$

\medskip
\noindent
In particular, we can apply this to the situation when $F = E$. Note that not any 
extension of $E_k$ by $\hbar^{k+1} E_s$ is a quantization 
to order
$\ell$. For example, this is not true for the direct sum extension (corresponding
to the zero extension class). More precisely, the extension gives a 
quantization of order $\ell$ when its class is mapped by 
$$
Ext^1_{\mathcal{O}_\hbar} (E_k, \hbar^{k+1} \cdot E_s) \to 
H^0(M, \hbar^{k+1}\cdot \mathcal{H}om_{\mathcal{O}_\hbar} (E_s, E_s)) 
$$
to the constant section $\hbar^{k+1} \cdot Id_{E_s}$ (this is a straightforward
computation since we know 
by now that locally $E_\ell$ is free over $\mathcal{O}_\hbar/\hbar^{\ell} \mathcal{O}_\hbar$). 
Since the local-to-global spectral sequence agrees with Yoneda products, both terms in 
the definition $O : = (\hbar^{k+1} \varphi_s) \circ e_\ell - f_\ell \circ 
\varphi_k$ are elements of  $Ext^1_{\mathcal{O}_\hbar}(E_k, \hbar^{k+1} \cdot F_s)$ 
that map to $\hbar^{k+1} \cdot \varphi_s \in H^0(M, \hbar^{k+1}\cdot \mathcal{H}om_{\mathcal{O}_\hbar} (E_s, E_s))$. 
Their difference maps to zero, and by the previous long exact 
sequence $O$ belongs to the image of 
$H^1(M, \hbar^{k+1} \cdot \mathcal{H}om_{\mathcal{O}_\hbar} (E_s, F_s)) \to 
Ext^1_{\mathcal{O}_\hbar} (E_k, \hbar^{k+1}\cdot F_s)$. 

Finally, suppose $a, b$ are as claimed in the proposition. Replacing extension classes
$e_\ell, f_\ell$ by $e_\ell+a, f_\ell + b$, respectively, we obtain new
order $\ell$ quantizations $E'_\ell, F'_\ell$ (since addition of $a$ does not 
change the image of the extension class in 
$H^0(M, \hbar^{k+1}\mathcal{H}om_{\mathcal{O}_\hbar} (E_s, E_s)) $, and 
similarly for $F_s$ and $b$). 
The new obstruction class
$O'$ is zero, hence the lift $\varphi_\ell$ of $\varphi_k$ exists
as a morphism $E'_\ell \to F'_\ell$. 

Finally, if $\varphi_\ell, \varphi'_\ell$ are two different extentions
then their post-compositions with the projection $F_\ell \to F_k$
descend from $E_\ell$ to $E_k$ and
agree with $\varphi_k$. 
Hence their difference may be viewed as a morphism 
$E_\ell \to \hbar^{k+1}\cdot F_s$. Such a morphism necessarily vanishes on 
multiples of $\hbar^{s+1}$ hence may be viewed as a
morphism $E_s \to \hbar^{k+1}\cdot E_s$ (and hence we can use the Kodaira
vanishing theorem). 
$\square$

\bigskip
\noindent
\textbf{Examples.} If the group $H^1(M, 
\mathcal{H}om_{\mathcal{O}_\hbar}(E_s, F_s))$ vanishes, then 
the obstruction $O$ must vanish too. From the spectral sequence 
associated with the filtration by powers of $\hbar$ we see 
that $O$ vanishes if $H^1(M, \mathcal{H}om_{\mathcal{O}_M}
(E, F)) = 0$. In particular this would be the case whenever $M$ 
is affine, or when $F$ is a twist of $E$ by an ample enough line bundle. 

\bigskip
\bigskip
\bigskip
\bigskip
\bigskip

\textit{Email address}: vbaranov@math.uci.edu

\bigskip

Department of Mathematics, UC Irvine, Irvine, CA 92697-3875, USA

\bigskip
\bigskip

\textit{Email address}: gh\,\_\,tufu@cosmology.name

\bigskip

Technology United for Ukraine, Kyiv, Ukraine

\end{document}